\documentclass[a4paper]{amsart}

\usepackage{amsfonts}
\usepackage{amssymb}
\usepackage[latin2]{inputenc}
\usepackage{amsmath}
\usepackage{amsthm}

\newtheorem{theorem}{Theorem}[section]
\newtheorem{lemma}[theorem]{Lemma}

\newtheorem{prop}[theorem]{Proposition}

\newtheorem{question}[theorem]{Question}

\theoremstyle{definition}   
 {           
\newtheorem{remark}[theorem]{Remark}

\newtheorem{notation}[theorem]{Notation}
}

\newcommand{\Z}{\mathbb{Z}}
\newcommand{\Q}{\mathbb{Q}}
\newcommand{\R}{\mathbb{R}}

\newcommand{\eps}{\varepsilon}

\def\beq{\begin{equation}}
\def\eeq{\end{equation}}
\def\beqq{\begin{equation*}}
\def\eeqq{\end{equation*}}

\newcommand{\norm}[1]{\left\lVert #1 \right\rVert} 
\newcommand{\abs}[1]{\left\lvert #1 \right\rvert} 
\newcommand{\sk}[1]{\langle #1 \rangle} 

\newcommand{\et}{\end{theorem}}
\newcommand{\el}{\end{lemma}}
\newcommand{\er}{\end{remark}}
\newcommand{\ed}{\end{defi}}
\newcommand{\bp}{\begin{proof}}
\newcommand{\ep}{\end{proof}}

\DeclareMathOperator{\dist}{dist}

\newcommand{\iA}{\mathcal{A}}

\renewcommand{\phi}{\varphi}

\title{Sets of large dimension not containing polynomial configurations}
\author{Andr\'as M\'ath\'e}
\address{Mathematics Institute, University of Warwick, 
Coventry, CV4 7AL, United Kingdom}
\email{A.Mathe@warwick.ac.uk}

\begin{document}

\begin{abstract}
The main result of this paper is the following. Given countably many multivariate polynomials with rational coefficients and maximum degree $d$, we construct a compact set $E\subset \R^n$ of Hausdorff dimension $n/d$ which does not contain finite point configurations corresponding to the zero sets of the given polynomials.

Given a set $E\subset \R^n$, we study the angles determined by three points of $E$. The main result implies the existence of a compact set in $\R^n$ of Hausdorff dimension $n/2$ which does not contain the angle $\pi/2$. (This is known to be sharp if $n$ is even.)  We show that there is a compact set of Hausdorff dimension $n/8$ which does not contain an angle in any given countable set. We also construct a compact set $E\subset \R^n$ of Hausdorff dimension $n/6$ for which the set of angles determined by $E$ is Lebesgue null. 

In the other direction, we present a result that every set of sufficiently large dimension contains an angle $\eps$ close to any given angle.

The main result can also be applied to distance sets.
As a corollary we obtain a compact set $E\subset \R^n$ ($n\ge 2$) of Hausdorff dimension $n/2$ which does not contain rational distances nor collinear points, for which the distance set is Lebesgue null, moreover, every distance and direction is realised only at most once by $E$.
\end{abstract}

\maketitle 

\section{Introduction}

The distance set conjecture asserts that for every analytic set $E\subset \R^n$ ($n\ge 2$) of Hausdorff dimension larger than $n/2$, the set of distances formed by $E$,
$$D(E) = \{\abs{x-y}\,:\,x,y\in E\}\subset \R,$$
has positive Lebesgue measure. The problem was first studied by Falconer \cite{Falc}. He showed that there is a compact set $E\subset \R^n$ of Hausdorff dimension $n/2$ for which $D(E)$ is Lebesgue null. In the other direction, he proved that $\dim_H E>(n+1)/2$ implies that $D(E)$ has positive Lebesgue measure.
According to the current best results, the same holds if
$\dim_H E> n/2 +1/3$ (see Wolff \cite{Wolff} and Erdo\u{g}an \cite{Erdogan}).
However, there are sequences of distances $(d_i)$, $d_i\to 0$, and compact sets $E\subset\R^n$ of dimension $n$ such that $d_i\notin D(E)$ for every $i$.

There has been recent interest in studying the set of angles $\iA(E)$ formed by three-point subsets of a given $E\subset\R^n$, see \cite{hetes, Harangi, IMP}. A particularly interesting open question is the following: What minimum dimension of a set $E\subset \R^n$  guarantees that $\iA(E)$ has positive Lebesgue measure?
A simple observation is that results for the distance set conjecture give results for the angle set as well. Therefore, for every analytic set $E\subset \R^n$ with $$\dim_H E>\min(n/2+4/3, \,n-1),$$ $\iA(E)$ has positive Lebesgue measure (Theorem~\ref{angledist}).
In the other direction, we construct a compact set $E\subset\R^n$ of Hausdorff dimension $\max(n/6, \, 1)$ such that $\iA(E)$ is Lebesgue null (Theorem~\ref{n/6}).

The following question was raised by Keleti. How large dimension can a Borel set $E\subset\R^n$ ($n\ge 2$) have such that $\iA(E)$ does not contain a given angle $\alpha$?

It follows from a theorem of Mattila that $\dim_H E>n-1$ implies $\iA(E)=[0,\pi]$, and that $\dim_H E > \lceil \frac{n}{2} \rceil$ implies $\pi/2\in\iA(E)$, see \cite{hetes}.
It is also known that if $E$ is an infinite set, then $E$ contains angles arbitrarily close to $0$ and $\pi$, see \cite{hetes}. It is proved in \cite{hetes}, that (independently of $n$) $\dim_H E>1$ guarantees angles arbitrarily close to $\pi/2$. Also, for some absolute constant $C$,  $\dim_H E > C \delta^{-1} \log \delta^{-1}$ implies that $E$ contains angles in the $\delta$ neighbourhood of $\pi/3$ and $2\pi/3$.

Harangi \cite{Harangi} proved that there is a self-similar set $E\subset\R^n$ of dimension $c_\delta n$ such that all angles formed by $E$ are in the $\delta$ neighbourhood of $\{0,\pi/3, \pi/2, 2\pi/3, \pi\}$, where $c_\delta$ is a constant independent of $n$. He also showed that there exists a compact set $E\subset \R^n$ with $\dim_H E = c\sqrt[3]{n}/\log n$ not containing the angle $\pi/3$ and $2\pi/3$.

The main result of this paper was motivated by the above question of Keleti. Instead of directly constructing sets of large dimension not containing a given angle, we provide a general method to tackle similar problems.
Given countably many multivariate polynomials $P_j$ in $nm_j$ variables with rational coefficients and maximum degree $d$, we construct a compact set $E\subset \R^n$ of Hausdorff dimension $n/d$ such that for every $j$ and for every $m_j$ distinct points
$$(x_{1,1}, \ldots, x_{1,n}), \,\ldots, \,(x_{m_j,1},\ldots, x_{m_j,n})\in E,$$
we have
$$P_j(x_{1,1},\ldots x_{1,n}, \,\ldots, \,x_{m_j,1}, \ldots, x_{m_j,n})\neq 0,$$
(Theorem~\ref{fo}).

For example, applying this result for the polynomial 
\beqq
\sk{y-x, \,z-x} = {\sum_{i=1}^n}
(y_i-x_i)(z_i-x_i)
\eeqq
 in $3n$ variables of degree $2$, we obtain a compact set $E\subset \R^n$
of Hausdorff dimension $n/2$ which does not contain the angle $\pi/2$ (Theorem~\ref{90}). By the result quoted from \cite{hetes}, this is sharp if $n$ is even. 

Another quick application of the main result is Theorem~\ref{rac}; there is a compact set $E\subset\R^n$ of Hausdorff dimension $n/4$ which does not contain angles $\alpha$ such that $\cos^2 \alpha$ is rational. Note that these include the angles $\pi/3$ and $2\pi/3$. Though the main result can be applied for polynomials with rational coefficients only, surprisingly it can be used to exclude `irrational' angles as well: there is a compact set $E\subset\R^n$ with $\dim_H E =n/8$ such that $E$ does not contain a given angle $\alpha\in[0,\pi]$, see Theorem~\ref{nper8}.

Our main result can also be used to strengthen Falconer's construction in \cite{Falc}.
We construct a compact set $E\subset\R^n$ with $\dim_H E =n/2$ such that the distance set, $D(E)$, is Lebesgue null, moreover, it satisfies the following peculiar properties:
$E$ does not contain rational distances nor collinear points, and every distance and direction is realised at most once by $E$
(Theorem~\ref{distkollin}). Note that the distance set conjecture implies that if $\dim_H E>n/2$, then there is a distance which is realised at least twice by $E$, see Remark~\ref{neminjektiv}.


Falconer \cite{FalcTriangle}, Keleti \cite{Keleti}, and Maga \cite{Maga} constructed sets of full dimension in $\R$ and in $\R^2$ which do not contain subsets similar to given three-point patterns. We show that a generalization of our main result implies the existence of such sets, see Section~\ref{fobbsection}.

Once the main theorem (Theorem~\ref{fo}) is understood, the sections of this paper can be read independently and in any order.
In Section~\ref{3} we prove the main result about excluding polynomial configurations. 
However, a necessary, technical but standard lemma is postponed until Section~\ref{8}. 
Section~\ref{4} contains the results about angles.
Section~\ref{5} contains results about collinearity and directions, and Section~\ref{6} about distance sets. 
Section~\ref{fobbsection} contains a generalization of Theorem~\ref{fo}.

In this paper, in Section~\ref{7}, we also present a result about finding many angles in sets of large dimension. We show that for every $\delta>0$ there is $c_\delta<1$ such that for every $n\ge 2$ and every Borel set $E\subset \R^n$ with $\dim_H E > c_\delta n$, $\iA(E)$ is $\delta$-dense in $[0,\pi]$ (Theorem~\ref{approxangles}). The proof combines a theorem of Mattila and a result about concentration of Lipschitz functions defined on the sphere (related to Dvoretzky's theorem). 

Finally, we pose open questions in Section~\ref{9}.

\begin{notation}
For $A\subset \R^n$, $x\in \R^n$ and $r>0$, $\dim_H A$ denotes the Hausdorff dimension of $A$, and 
$$B(x,r)=\{y\in\R^n: |x-y|<r\}$$
and
$$B(A,r)=\{y\in\R^n: \text{there is }x\in A \text{ with } |x-y|<r\}.$$

The angle set of $E\subset\R^n$ is defined as
$$\iA(E)=\{\angle yxz \, : \, x,y,z \text{ are distinct points of } E\}.$$
\end{notation}

\section{Main result}\label{3}

Let $P$ be a real-valued polynomial in $nm$ variables. Let $x_1, \ldots, x_m \in \R^n$. By writing $P(x_1, \ldots, x_m)$, we mean $P(x_{1,1}, \ldots, x_{1,n}, \ldots, x_{m,1}, \ldots, x_{m,n})$. We will denote the derivative by $P'$, which is a map $\R^{nm}\to\R^{nm}$. Partial derivative in direction $x_{i,j}$ will be denoted by $\partial_{i,j} P$.

\begin{lemma}\label{folemma}
Let $P$ be polynomial of degree $d$ in $nm$ variables with rational coefficients. Let $x_1, \ldots, x_m$ be $m$ distinct points of $\R^n$ satisfying $P(x_1, \ldots, x_m)=0$ and $P'(x_1, \ldots, x_m)\neq 0 \in \R^{nm}$. Then there exists $r>0$ such that for every sufficiently small $h>0$ there is a set $E_h\subset \R^n$ for which $B(E_h, h)=\R^n$
and 
$P(y_1, \ldots, y_m)\neq 0$ whenever $y_i\in B(x_i,r)\cap B(E_h, h^d/(\log h^{-1}))$ ($i=1,\ldots,m$).
\end{lemma}

It is worth to think about $E_h$ as a locally finite set in which the distance of ``neighbouring'' points is approximately $h$.

\begin{proof}
We may suppose without loss of generality that $P$ has integer coefficients only, and that $\partial_{1,1} P (x_1, \ldots, x_m) = a \neq 0$. Let $r>0$ be so small that
\beq\label{amekkora}
|\partial_{1,1} P (y_1, \ldots, y_m) -a|\le a/3
\eeq
whenever $y_i\in B(x_i,r)$, and that $r\le \min_{i\neq j} |x_i-x_j|/4$.

Let $N$ be a sufficiently large positive integer. Let $L=N^{-1}\Z^d$.
Let $$A_i=B(x_i, r) \cap L \quad (i=2,\ldots, m).$$
Let $u=(1/(2N^d a), 0, \ldots) \in R^{n}$.
Let
$$A_1=B(x_1, r) \cap L + u.$$
Suppose that $y_i\in B(x_i,r)\cap L$ ($i=1,\ldots, m$). Then $P(y_1,\ldots y_m)$ is of the form $j/N^d$ for some integer $j$ with $|j|\le C N^d$, where $C$ is independent of $N$, and only depends on $P$ and the points $x_i$.

Let $y'_1=y_1+u \in A_1$. Then
$$\abs{P(y'_1, y_2, \ldots, y_m)-P(y_1, \ldots, y_m) - |u| \partial_{1,1}P(y_1, \ldots, y_m)} \le C'|u|^2,$$
where $C'$ comes from the second derivative of $P$, independent of the choice of $y_i$. 
Using \eqref{amekkora}, 
$$\abs{P(y'_1, y_2, \ldots, y_m) - \frac{j}{N^d}- \frac{1}{2N^d}} \le  \frac{1}{6N^d} + C'\frac{1}{4N^{2d} a^2},$$
end if we choose $N$ large enough,
\beq\label{racs1}
\abs{P(y'_1, y_2, \ldots, y_m) - \frac{j}{N^d} - \frac{1}{2N^d}} \le  \frac{1}{5N^d}.
\eeq

Suppose now that $z_i\in B(A_i, c N^{-d})$ ($i=1,\ldots, m$) for some small constant $c>0$. By \eqref{racs1}, there is an integer $j$ for which
$$
\abs{P(z_1, \ldots, z_m) - \frac{j}{N^d} - \frac{1}{2N^d}} \le  \frac{1}{5N^d} + C''cN^{-d},
$$
where $C''$ only depends on $P$ and the points $x_i$. Choosing $c=\frac{1}{20C''}$, we have
$$
\abs{P(z_1, \ldots, z_m) - \frac{j}{N^d} - \frac{1}{2N^d}} \le  \frac{1}{4N^d},
$$
that is, 
\beq\label{racs3}
\abs{N^d P(z_1, \ldots, z_m) - j - \frac{1}{2}} \le  \frac{1}{4}
\eeq
for some integer $j$. Therefore the distance of $N^d P(z_1, \ldots, z_m)$ to the nearest integer is at least $1/4$, and $P(z_1, \ldots, z_m)\neq 0$.

Let $h>0$ be sufficiently small, and choose the smallest possible $N$ so that $h \ge \sqrt{n} N^{-1}$. Let 
$$E_h= \bigcup_{i=1}^m A_i \cup \left(\R^n \setminus \bigcup B(x_i,r)\right).$$
Then $\R^n = B(E_h, \sqrt{n}N^{-1}) = B(E_h, h)$. If $h$ is sufficiently small, then
$$h^d/(\log h^{-1})\le cN^{-d}= N^{-d}/(20C''),$$
and \eqref{racs3} implies that $P(z_1, \ldots, z_m)\neq 0$ whenever
\beqq
z_i \in B(x_i,r/2) \cap B(E_h, h^d/(\log h^{-1})) \quad (i=1,\ldots, m). \qedhere
\eeqq
\end{proof}

\begin{prop}\label{foprop}
Let $n\ge 1$. Let $J$ be a countable set. For each $j\in J$, let $m_j$ be a positive integer, and let $P_j:\R^{nm_j}\to\R$ be a polynomial in $nm_j$ variables with rational coefficients. Assume that $d$ is the maximum degree of the polynomials $P_j$ ($j\in J$).
Then there exists a compact set $E\subset \R^n$ of Hausdorff dimension $n/d$ such that for every $j\in J$, $E$ does not contain $m_j$ distinct points $x_1, \ldots, x_{m_j}$ satisfying $P_j(x_1, \ldots, x_{m_j})=0$ and $P_j'(x_1, \ldots, x_{m_j})\neq 0$.
\end{prop}
\begin{proof}
For $j\in J$ let
\begin{align*}
G_j=\{(x_1, \ldots, x_{m_j}) \in \R^{nm_j} : \ & P_j(x_1, \ldots, x_{m_j})=0, \,P'_j(x_1,\ldots x_{m_j})\neq 0, \\
& x_i\neq x_j \text{ if } i\neq j\}.
\end{align*}
For every point $x\in G_j$ we may apply Lemma~\ref{folemma}. We obtain $0<r=r_j(x)$ such that for every sufficiently small $h>0$ there is a set $E_h\subset \R^n$ for which
$B(E_h,h)=\R^n$ and $B(E_h, h^d/(\log h^{-1}))$ does not contain distinct points $y_1, \ldots, y_{m_j}$ satisfying $P(y_1, \ldots, y_{m_j})=0$ and $(y_1, \ldots, y_{m_j})\in B(x, r_j(x))$.

The sets $B(x,r_j(x))$ ($x\in G_j$) form an open cover of $G_j\subset \R^{n m_j}$.
The space $\R^{n m_j}$ is second countable, so every subset has the Lindel\"of property.
Therefore we can choose a countable open subcover for $G_j$. If we do this for all $j\in J$, we still have altogether countably many open balls (in various dimensions). Therefore we may fix a function
$$\phi : \{1,2,\ldots\}\to \bigcup_{j\in J} \{j\} \times G_j$$
such that for every $j\in J$,
\beq\label{mindentfed}
\bigcup \big\{ B\big(\phi_2(i), r_j(\phi_2(i))\big) : \phi_1(i)=j, \ i\in \{1,2,\ldots\}\big\} \supset G_j
\eeq
where $\phi(i)=(\phi_1(i), \phi_2(i))$. 

Choose $h_1$ small enough so that it satisfies the assumption \eqref{tul1} of Lemma~\ref{dimlemma}, and also small enough that the statement of Lemma~\ref{folemma} holds for the polynomial $P_{\phi_1(1)}$, point $(x_1,\ldots, x_m)=x=\phi_2(1)$ and $r=r_{\phi_1(1)}(\phi_2(1))$. We obtain a set $E_{h_1}\subset \R^n$. For $j\ge 2$, if $h_i$ ($i< j$) is already chosen, choose $h_j$ 
small enough to satisfy \eqref{tul1} and \eqref{tul2}, and that Lemma~\ref{folemma} holds for the polynomial $P_{\phi_1(j)}$, point $(x_1,\ldots, x_m)=x=\phi_2(j)$ and $r=r_{\phi_1(j)}(\phi_2(j))$. We obtain sets $E_{h_i}\subset \R^n$ ($i=1,2,\ldots$).

Let $$E=B(0,1)\cap \bigcap_{i=1}^\infty \overline{B(E_{h_i}, h_i^d/(2\log h_i^{-1}))}.$$
Then $E\subset \bigcap B(E_{h_i}, h_i^d/(\log h_i^{-1}))$. Thus \eqref{mindentfed} and Lemma~\ref{folemma} imply that for every $j\in J$, $E$ does not contain $m_j$ points $x_1, \ldots, x_{m_j}$ such that $(x_1,\ldots,x_{m_j})\in G_j$. In other words, $E$ does not contain $m_j$ distinct points $x_1, \ldots, x_{m_j}$ satisfying $P_j(x_1, \ldots, x_{m_j})=0$ and $P'_j(x_1,\ldots, x_{m_j})\neq 0$.

Lemma~\ref{dimlemma} implies that the Hausdorff dimension of $E$ is at lest $n/d$.
\end{proof}


We can get rid of the assumption on the derivatives $P'_j$ in Proposition~\ref{foprop} by considering a larger family of polynomials.

\begin{theorem}\label{fo}
Let $n\ge 1$. Let $J$ be a countable set.
For each $j\in J$, let $m_j$ be a positive integer, and let
$P_j:\R^{nm_j}\to\R$ be a (non identically zero) polynomial in $nm_j$ variables with rational coefficients. Assume that $d$ is the maximum degree of the polynomials $P_j$ ($j\in J$).
Then there exists a compact set $E\subset \R^n$ of Hausdorff dimension $n/d$ such that for every $j\in J$, $E$ does not contain $m_j$ distinct points $x_1,\ldots, x_{m_j}$ satisfying $P_j(x_1, \ldots, x_{m_j})=0$.
\end{theorem}

\begin{proof}
Consider a polynomial $P=P_j$ ($j\in J$). Let $y_1, \ldots, y_{nm}$ denote the variables of $P$. Choose a monomial of $P_j$ which has the same degree as $P$ itself. Let this be $y_{i_1} y_{i_2} \cdots y_{i_r}$, where $i_k\in \{1,\ldots, nm\}$ for $k=1,\ldots, r$.
Consider the partial derivatives
$$\partial_{i_1} P, \ \partial_{i_2} \partial_{i_1} P, \ \ldots, \ \partial_{i_r}\cdots\partial_{i_1}P.$$
The last polynomial in this list is constant and non-zero.
Let $x_1, \ldots, x_m\in \R^n$ be $m$ distinct points. Notice that if $P(x_1,\ldots, x_m)=0$, then there exists an integer $1\le k\le r$ such that
\beq\label{obs}
\partial_{i_{k-1}} \cdots \partial_{i_1} P(x_1, \ldots, x_m)=0, \quad
\partial_{i_{k}} \cdots \partial_{i_1} P(x_1, \ldots, x_m)\neq 0.
\eeq
We apply Proposition~\ref{foprop} for these polynomials $P, \ \partial_{i_1}P, \ \ldots, \ \partial_{i_{r-1}}\cdots\partial_{i_1}P$, for every $P=P_j$ ($j\in J$). By the observation \eqref{obs}, the obtained set $E$ has the desired properties.
\end{proof}

\begin{remark}\label{fosharp}
For every positive integer $d$, there is a polynomial of degree $d$ for which Theorem~\ref{fo} is sharp. By a theorem of Mattila (see \cite{Mattila}), every analytic set $E\subset \R^d$ with $\dim_H E>1$ contains at least $d+1$ points in a hyperplane (in fact, there is a hyperplane which intersects $E$ is a set of positive dimension).
On the other hand, $x_0, \ldots, x_{d}\in\R^d$ lie in a hyperplane if and only if the determinant of the $d\times d$ matrix formed by the vectors $x_i-x_0$ ($i=1,\ldots, d$) is zero. This determinant is a polynomial in $d(d+1)$ variables of degree $d$. Therefore Theorem~\ref{fo} applied to this polynomial with $n=d$ gives a sharp result.
\end{remark}

\begin{remark}\label{fonemsharp}
Theorem~\ref{fo} is obviously not sharp for every family of polynomials. For example, let $P(x,y,z,u,v,w)=xyz-uvw$. Then the theorem yields a compact set $E\subset\R$ of dimension $1/3$ such that $P(x,y,z,u,v,w)\neq 0$ whenever $x,y,z,u,v,w\in E$ are distinct. However, there is such set $E$ of dimension $1$ as well. To see this, apply the theorem for $Q(x,y,z,u,v,w)=x+y+z-u-v-w$, the compact set $F\subset \R$ obtained has dimension $1$, and put $E=\{e^x: x\in F\}$.
\end{remark}

Remark~\ref{fonemsharp} indicates that considering a diffeomorphism (between subsets) of $\R^n$ with Theorem~\ref{fo} can produce stronger results. In Section~\ref{fobbsection} we sketch a generalization of Theorem~\ref{fo} which allows the use of multiple diffeomorphisms, and we give an application.

\section{Angle sets}\label{4}

\begin{theorem}\label{90}
Let $n\ge 2$. There exists a compact set $E\subset \R^n$ of Hausdorff dimension $n/2$ such that $E$ does not contain three points forming the angle $\pi/2$. Moreover, $E$ does not contain four points $x,y,z,v$ such that $x-y$ and $z-v$ are orthogonal.
\end{theorem}

\begin{proof}
It is enough to exclude solutions of $\sk{x-y, z-v}=0$ when $x, y, z$ are three distinct points in $\R^n$ and $z\neq v$.
Applying Theorem~\ref{fo} for the polynomials $P_1(x,y,z,v)=\sum_{i=1}^n (x_i-y_i)(z_i-v_i)$ and $P_2(x,y,z)=\sum_{i=1}^n (x_i-y_i)(z_i-x_i)$ yields the result.
\end{proof}

\begin{theorem}\label{rac}
Let $n\ge 2$. There exists a compact set $E\subset \R^n$ of Hausdorff dimension $n/4$ such that $E$ does not contain three points forming an angle $\alpha$ for which $\cos^2 \alpha$ is rational. (Moreover, $E$ does not contain four distinct points $x,y,z,v$ such that the directions $x-y$ and $z-v$ form an angle $\alpha$ for which $\cos^2 \alpha$ is rational.)
\end{theorem}
\begin{proof}
There are countably many $\alpha$ for which $\cos^2 \alpha$ is rational. For such $\alpha$,
consider the polynomials $P_\alpha:\R^{3n}\to \R$ and $Q_\alpha:\R^{4n}\to\R$,
$$P_\alpha(x,y,z)=\sk{y-x, z-x}^2-\cos^2\alpha |y-x|^2 |z-x|^2,$$
$$Q_\alpha(x,y,z,v)=\sk{y-x, z-v}^2-\cos^2\alpha |y-x|^2 |z-v|^2.$$
Applying Theorem~\ref{fo} for these polynomials yields the statement.
\end{proof}

\begin{remark}
In the proof of Theorem~\ref{rac}, when $\alpha$ is not $0$, $\pi/2$ or $\pi$,  we could just refer to Proposition~\ref{foprop}. This is because the derivatives of $P_\alpha$, $Q_\alpha$ do not vanish where the polynomials vanish. However, for the angles $0$, $\pi/2$ or $\pi$ referring to Theorem~\ref{fo} is necessary: this corresponds to the fact that for $\pi/2$ and for $0, \pi$ we have better results than dimension $n/4$ (Theorem~\ref{90} and Theorem~\ref{kollin}).
\end{remark}

Though we allow polynomials with rational coefficients only, Theorem~\ref{fo} can be used to exclude irrational angles as well.

\begin{theorem}\label{nper8}
Let $S\subset [0,\pi]$ be a given countable family of angles. There exists a compact set $E\subset \R^n$ of Hausdorff dimension $n/8$ such that $E$ does not contain three points forming an angle $\alpha\in S$. (Moreover, $E$ does not contain four distinct points $x,y,z,v$ such that the directions $x-y$ and $z-v$ form an angle $\alpha\in S$.)
\end{theorem}
\begin{proof}
We apply Theorem~\ref{fo} for two polynomials of degree 8, $P:\R^{8n}\to\R$ and $Q:\R^{6n}\to\R$, where
$$P(a,b,c,d,e,f,g,h)=\sk{a-b,c-d}^2 \abs{e-f}^2 \abs{g-h}^2 - \sk{e-f,g-h}^2\abs{a-b}^2 \abs{c-d}^2$$
$$Q(a,b,c,e,f,g)=\sk{a-b,c-b}^2 \abs{e-f}^2 \abs{g-f}^2 - \sk{e-f,g-f}^2\abs{a-b}^2 \abs{c-b}^2.$$
We obtain a compact set $E\subset \R^n$ of Hausdorff dimension $n/8$ such that
there are no distinct points $a,b,c,d,e,f,g,h$ in $E$ satisfying any of the equations
\beq\label{8as}
\frac{\sk{a-b, c-d}}{\abs{a-b}\abs{c-d}} = \frac{\sk{e-f, g-h}}{\abs{e-f}\abs{g-h}}
\eeq
or
\beq\label{6os}
\frac{\sk{a-b, c-b}}{\abs{a-b}\abs{c-b}} = \frac{\sk{e-f, g-f}}{\abs{e-f}\abs{g-f}}.
\eeq

For each $\alpha$, $E$ may contain three points $a,b,c$ forming the angle $\alpha$, and/or four distinct points $a,b,c,d$ for which $a-b$ an d $c-b$ form the angle $\alpha$. However, if we remove these points (at most $3+4$) from $E$, then the remaining set will not contain three points forming the angle $\alpha$ nor four distinct points determining two directions forming the angle $\alpha$.

As these polynomials have degree $8$, we obtain a compact set $E\subset \R^n$ of Hausdorff dimension $n/8$. This set may contain angles belonging to $S$. But as \eqref{8as} and \eqref{6os} have no solutions in $E$, if we remove at most $4+3$ points from $E$ for each $\alpha\in S$, then $E$ will not contain three points forming an angle $\alpha\in S$ or four distinct points forming two directions of angle $\alpha\in S$. So if $T$ is the countable set we remove (considering all $\alpha\in S$), then $E\setminus T$ satisfies all properties of the theorem except for compactness.
Notice that, by Remark~\ref{dimremark}, $E$ and $E\setminus T$ has positive Hausdorff measure with respect to the gauge function $\phi(r)= r^{n/8} (\log r^{-1})^{n+1}$. By Frostman lemma (see \cite{Mattila}), we can find a compact subset of $E\setminus T$ which has positive $\phi$-measure, and thus, is of dimension $n/8$.
\end{proof}

Note that for $[0,1]\subset\R^n$, the angle set $\iA([0,1])=\{0,\pi\}$.

\begin{theorem}\label{n/6}
Let $n\ge 2$. There exists a compact set $E$ of Hausdorff dimension $n/6$ for which the angle set $\iA(E)$ has zero Lebesgue measure.
\end{theorem}

The construction is similar to Falconer's original construction of a set of dimension $n/2$ for which the distance set is Lebesgue null \cite{Falc}. However, the argument here is more complicated.
\begin{proof}
Let $N_i$ be a rapidly increasing sequence of positive integers.
For $x\in \R$, let $\norm{x}$ denote the distance to the nearest integer.

Let $A_i=\{x\in [0,1] \, : \, \norm{xN_i}\le N_i^{-6}/i^4 \}$ (the power of $i$ is chosen for convenience). Consider products of the set $A_i$ and let $$B_i=A_i\times \cdots \times A_i\subset \R^n.$$
Suppose that $a,b,c,d\in B_i$. Then $\sk{a-b,c-d}\in[-n,n]$ lies in the $nN_i^{-6}/i^4$ neighbourhood of some $j/N_i^2$ where $j$ is an integer with $|j| \le N_i^2 n$.
Therefore 
$$\sk{B_i-B_i, B_i-B_i}= \{\sk{a-b,c-d}\,:\,a,b,c,d\in B_i\}$$ can be covered by $2N_i^2n+1\le 3N_i^2 n$ many intervals of length $2nN_i^{-6}/i^4$. As $\sk{B_i,B_i}\subset [-n,n]$, a  short calculation yields that
$$C_i=\left\{\frac{\sk{a-b,c-b}}{\sqrt{\abs{\sk{a-b,a-b}}} \sqrt{\abs{\sk{c-b,c-b}}}} \,:\, a,b,c\in B_i, \,
|a-b|^2\ge 1/i, \, |c-b|^2\ge 1/i \right\}$$
can be covered by $(3N_i^2n)^3$ many intervals of length $3 ni^2 (2nN_i^{-6}/i^4)$. Thus \beq\label{mertekCi}
\lambda(C_i) \le 162 n^5 /i^2.
\eeq

Let $B=\bigcap_{i=1}^\infty B_i$. Standard arguments imply that the Hausdorff dimension of $B$ is $n/6$ (or see Lemma~\ref{dimlemma}).

Let $x,y,z\in B$ be three distinct points, let $\alpha=\angle yxz$.
Then $$\cos\alpha= \frac{\sk{y-x, z-x}}{\abs{y-x}\abs{z-x}} \in \bigcap_{k=1}^\infty \bigcup_{i=k}^\infty C_i,$$
and $\bigcap \bigcup C_i$ is Lebesgue null by \eqref{mertekCi}. Therefore $\iA(B)$ is Lebesgue null.
\end{proof}

Any result for the distance set problem gives a result for the angle set problem as well. Independently from the author, this was observed also by Iosevich, Mourgoglou and Senger. 

\begin{theorem}\label{angledist}
Let $E\subset\R^n$ be a Borel set with $\dim_H E>n/2+4/3$. 
Then $\iA(E)$ has positive Lebesgue measure.
\end{theorem}
\begin{proof}
Take a radial projection of $E$ to a sphere $S$ of unit radius with center $x\in E$, and let $F\subset S$ be the set obtained. Clearly $\dim_H F\ge \dim_H E -1$, and thus the distance set $D(F)$ has positive Lebesgue measure by the results of Wolff and Erdo\u{g}an.
If $F$ contains two points $y,z$ with distance $d$, then $\alpha=\angle yxz$ 
satisfies $d=2\sin(\alpha/2)$, and $\alpha \in \iA(E)$. Therefore $\iA(E)$ has positive Lebesgue measure. 
\end{proof}

\section{Collinearity and directions}\label{5}
\begin{theorem}\label{kollin}
Let $n\ge 2$. There exists a compact set $E\subset \R^n$ of Hausdorff dimension $n/2$ such that $E$ does not contain three collinear points, moreover, every direction is realised by $E$ at most once.
\end{theorem}
\begin{proof}
Let $x,y,z,v\in\R^n$, and consider the polynomials $P:\R^{4n}\to \R$, $Q:\R^{3n}\to\R$,
$$P(x,y,z,v)=(y_2-x_2)(v_1-z_1)-(y_1-x_1)(v_2-z_2),$$
$$Q(x,y,z)=(y_2-x_2)(y_1-z_1)-(y_1-x_1)(y_2-z_2).$$
If $y-x$ and $v-y$ are parallel, then their projection to the plane spanned by the first two coordinate axes are also parallel, and thus $P(x,y,z,v)=0$. (We consider the $0$ vector to be parallel to every vector.) Similarly, if $x,y,z$ are collinear, $Q(x,y,z)=0$. Applying Theorem~\ref{fo} for $P$ and $Q$ we obtain the result.
\end{proof}

This statement can be strengthened.

\begin{theorem}\label{kollin2}
Let $n\ge 2$. There exists a compact set $E\subset \R^n$ of Hausdorff dimension $n-1$ such that $E$ does not contain three collinear points, moreover, every direction is realised by $E$ at most once.
\end{theorem}

Note that the unit sphere $S^{n-1}\subset \R^n$  satisfies the first property: it has dimension $n-1$ and does not contain three collinear points.

\begin{proof}[Sketch of proof.]
For $n=2$ we are done by Theorem~\ref{kollin}. Assume that $n\ge 3$. Let $E_2\subset\R^2$ be the compact set obtained from Theorem~\ref{kollin}. 

It is well known that there is a continuous function $f:[0,1]\to \R$ such that the graph of $f$ (a subset of $\R^2$) has Hausdorff dimension $2$.
 (In fact, Theorem~\ref{fo} also implies the existence of such function by using the polynomial $P(x,y)=y_1-x_1$, and then applying Tietze's extension theorem.)

A straightforward generalization is that for every uncountable compact set $A\subset \R^k$ there is a continuous function $f:A\to \R^m$ such that the graph, a compact subset of $\R^{k+m}$, has Hausdorff dimension $\dim_H A + m$. We apply this result for $k=2$, $A=E_2$, $m=n-2$. We obtain a compact set $E\subset \R^n$ of dimension $n-1$ such that $E$ projects injectively to $E_2\subset \R^2$ (spanned by the first two coordinate axes). Then $E$ does not contain three collinear points, nor four distinct points $x,y,z,v$ such that $y-x$ and $v-z$ are parallel.
\end{proof}

\begin{remark}
Theorem~\ref{kollin2} is sharp. For every Borel set $A\subset \R^n$ with $\dim_H A>n-1+s$ there are lines $L$ such that $\dim_H (A\cap L) \ge s$ (\cite{Mattila}).
\end{remark}

\section{Distance sets}\label{6}

\begin{theorem}\label{dist1}
Let $n\ge 1$. There exists a compact set $E\subset\R^n$ of Hausdorff dimension $\max(1, n/2)$ such that the distance set is a nullset, $E$ does not contain rational distances, moreover, every distance is realised at most once by $E$.
\end{theorem}

\begin{proof}
First let $n\ge 2$. Consider the polynomials (of degree 2)
$$P_r(a,b,c,d)=|a-b|^2-|c-d|^2-r \quad (r\in\Q), $$
$$P^*_r(a,b,c)=|a-b|^2-|c-b|^2-r \quad (r\in\Q), $$
$$Q_r(a,b)=|a-b|^2-r \quad (r\in\Q),$$
and apply Theorem~\ref{fo}. The obtained set $E$ does not contain rational distances  (except for zero) because of $Q_r$, and every (nonzero) distance is realised at most once because of $P_0$ and $P^*_0$.
Let $D^2(E)$ denote the squares of the distances realised by $E$. The polynomials $P_r$ and $P^*_r$ imply that $D^2(E)-D^2(E)$ does not contain rational numbers (except for $0$). As $D^2(E)$ is Lebesgue measurable, Steinhaus theorem implies that $D^2(E)$ and thus $D(E)$ are nullsets. 

For $n=1$, one can take the same polynomials without squaring the distances.
$$P_r(a,b,c,d)=(a-b)-(c-d)-r \quad (r\in\Q), $$
$$P^*_r(a,b,c)=(a-b)-(b-c)-r \quad (r\in\Q), $$
\begin{equation*}
Q_r(a,b)=(a-b)-r \quad (r\in\Q). \qedhere 
\end{equation*}
\end{proof}

In Theorem~\ref{dist1}, one may exclude any given countable set from $D(E)$ the same way as we did in Theorem~\ref{nper8}. 

\begin{remark}\label{neminjektiv}
The distance set conjecture implies that for every analytic set $E\subset\R^n$ with $\dim_H E > n/2$, there is a distance which is realised at least twice by $E$. In fact, there are four distinct points $x,y,x',y'\in E$ such that $|x-y|=|x'-y'|>0$. To see this, choose uncountable many disjoint Borel sets $E_i\subset E$ with $\dim_H E_i=(n/2+\dim_H E)/2$. If the distance set conjecture holds, $D(E_i)$ has positive Lebesgue measure. There are no uncountably many disjoint measurable sets of positive measure, thus the statement follows.
\end{remark}

Combining the polynomials in the proofs of Theorem~\ref{dist1} and Theorem~\ref{kollin} we obtain the result stated in the abstract.
\begin{theorem}\label{distkollin}
Let $n\ge 2$. There exists a compact set $E\subset\R^n$ of Hausdorff dimension $n/2$ such that the distance set is a nullset, $E$ does not contain rational distances nor collinear points, moreover, every distance and direction is realised at most once by $E$.
\end{theorem}



\section{Sets not containing similar copies of given patterns}\label{fobbsection}
It is straightforward to generalize Theorem~\ref{fo} the following way.
 
\begin{theorem}\label{fobb}
Let $n\ge 1$. Let $J$ be a countable set. For each $j\in J$, let $m_j$ be a positive integer, let $P_j:\R^{nm_j}\to\R$ be a (non identically zero)  polynomial in $nm_j$ variables with rational coefficients, and let $\Phi_{j,1},\ldots, \Phi_{j,m_j}$ be $C^1$ smooth diffeomorphisms of $\R^n$. Assume that $d$ is the maximum degree of the polynomials $P_j$ ($j\in J$).
Then there exists a compact set $E\subset \R^n$ of Hausdorff dimension $n/d$ such that for every $j\in J$, $E$ does not contain $m_j$ distinct points $x_1,\ldots, x_{m_j}$ satisfying $$P_j(\Phi_{j,1}(x_1), \ldots, \Phi_{j,m_j}(x_{m_j}))=0.$$
\end{theorem}

For example, one can choose the diffeomorphisms to be multiplication by invertible matrices. Note that the entries of the matrices can be rational or irrational, but the coefficients of the polynomials are still required to be rational numbers.

Using Theorem~\ref{fobb}, we immediately get an alternative proof to theorems of Falconer, Keleti and Maga about full dimensional sets in $\R$ and $\R^2$ not containing similar copies of given triangles.

\begin{theorem}[Falconer \cite{FalcTriangle}, Keleti \cite{Keleti}, Maga \cite{Maga}]\label{FKM}
Let $n=1$ or $2$. Let $J$ be a countable set. For each $j\in J$, let $T_j$ be a three-point subset of $\R^n$.  There is a compact set $E\subset\R^n$ of Hausdorff dimension $n$ which does not contain a subset which is similar to any of the sets $T_j$.
\end{theorem}
\begin{proof}
For each $j\in J$, there is an $n\times n$ invertible matrix $A_j$ (a similarity) such that $\{a,b,c\}\subset \R^n$ is similar to $T_j$ if and only if $c-a=A_j(b-a)$, that is,
$$(A_j-I)a-A_j b +c=0.$$
We choose the diffeomorphisms $\Phi_{j,1}, \Phi_{j,2}, \Phi_{j,3}$ to be multiplication by $A_j-I$, $-A_j$, and $I$.
If $n=1$, we put $P_j(x,y,z)=x+y+z$; if $n=2$, we put $P_j(x_1, x_2, y_1, y_2, z_1, z_2)=x_1+y_1+z_1$. The compact $E$ set obtained by Theorem~\ref{fobb} satisfies all requirements.
\end{proof}

\begin{remark}
It is not known what dimension a compact set in $\R^3$ can have if it does not contain three vertices of an equilateral triangle, or four vertices of a regular tetrahedron. There is no non-trivial linear equation satisfied by all such three of four points, therefore the proof we gave for Theorem~\ref{FKM} does not work if $n=3$ (or if  $n\ge 3$). On the other hand, from Theorem~\ref{fo} it follows that there is a compact set $E\subset\R^n$ of dimension $n/2$ which does not contain three vertices of an equilateral triangle: one can use, for example, the polynomial $\sk{y-x, \,z-\frac{x+y}{2}}$.
\end{remark}

\section{Finding angles in sets of large dimension}\label{7}

\begin{theorem}\label{approxangles}
Let $0<\eps\le 1$. Let $E\subset \R^n$ be a Borel set such that
\beq\label{konc1}
\dim_H E > \left\lfloor \left(1-\frac{\eps^2}{8 \log (8/\eps)}\right)\cdot n \right\rfloor+1.
\eeq
Then for every $\alpha\in[0,\pi]$, $E$ contains three points $x,y,z$ such that $\angle xyz \in (\alpha-\pi \eps, \alpha+\pi \eps).$
\end{theorem}

The proof combines a theorem of Mattila about dimension of plane sections, and a result about concentration of Lipschitz functions defined on the sphere.

\def\med{\text{med}}
\begin{proof}
Let $E\subset \R^n$ a Borel set of dimension $s$, and let $m<s$ be an integer.
By a theorem of Mattila \cite[Theorem 10.11.]{Mattila}, there exists a point $x\in E$ such that almost every $n-m$ dimensional plane through $x$ intersects $E$ in a set of dimension $s-m$. Let $S$ be the unit sphere with center $x$, and project $E\setminus \{x\}$ to $S$; let $K\subset S$ be the set obtained.
Then almost every $\lfloor n-s+1 \rfloor$ dimensional plane through $x$ intersects $K$.

If $K$ contains two points of distance $2\sin (\alpha/2)$, then $E$ contains three points forming the angle $\alpha$.

\def\cK{\overline{K}}
Let $\sigma$ denote the spherical probability measure on $S$.
Let $\cK_\eps=S\cap \overline{B(K,\eps)}$ be the closure of the $\eps$ neighbourhood of $K$ in $S$.

First let us suppose that $\sigma(\cK_\eps)>1/3$. Let $0<t<2$.
By \eqref{konc1} we may assume that $n\ge 3$ and thus we can choose three points in $S$ forming an equilateral triangle of side $t$. Considering a random rotated image of these three points, with positive probability, at least two of them are in $\cK_\eps$. Therefore $\cK_\eps$ contains every distance $0\le t<2$. Using $\eps\le 1$ and the remark above, this implies that $E$ contains an angle in every interval $(\alpha - \pi \eps, \alpha+\pi\eps)$.

Now let us suppose  that $\sigma(\cK_\eps)\le 1/3$.
Define $f:S\to \R$ by $f(y)=\dist(y, K)$. Then $f$ is Lipschitz with Lipschitz constant $1$. The median of $f$, defined as $$\med(f)=\sup \{t: \sigma(\{y : f(y)\le t\})\le 1/2\},$$ is more than $\eps$, as $\sigma(\{y : f(y)\le \eps\})=\sigma(\cK_\eps) \le 1/3$.
Therefore, by \cite[14.3.4 Proposition]{Matousek}, there is a plane $L$ through $x$, such that all values of $f$ restricted to $S\cap L$ are in $[\med(f)-\eps, \med(f)+\eps]$ and
$$\dim L \ge \frac{\eps^2}{8 \log (8/\eps)}\cdot n -1.$$
Specially, $S\cap L$ is disjoint from the closure of $K$.
The neighbouring planes have the same property, and thus, by the above theorem of Mattila, we must have
$$ \frac{\eps^2}{8 \log (8/\eps)}\cdot n -1 \le \dim L \le \lfloor n-s+1 \rfloor -1,$$
so
$$s\le \left\lfloor \left(1-\frac{\eps^2}{8 \log (8/\eps)}\right)\cdot n \right\rfloor+1.$$
This cannot happen if \eqref{konc1} holds.
\end{proof}

\section{Lower bound for the dimension}\label{8}

Here we prove the statement which we needed in Section~\ref{3}. Similar arguments can be found in \cite[Theorem 8.15]{Falconer}. Note that the exact form of the conditions \eqref{tul1} and \eqref{tul2} is not important for the application in Section~\ref{3}.

\begin{lemma}\label{dimlemma}
Let $n\ge 1$ and $d\ge 1$ be integers. Let $h_i>0$ be a rapidly decreasing sequence such that
\beq\label{tul1}
h_1\le 1/10, \quad h_{j}\le h_{j-1}^{d+1}/6 \quad (j=2, 3, \ldots),
\eeq
and
\beq\label{tul2}
h_j\le \exp\left(-\prod_{i=1}^{j-1} h_i^{-(dn+1)}\right) \quad (j=2, 3, \ldots).
\eeq
 Suppose that $E_i\subset \R^n$ satisfies $B(E_i, h_i)=\R^n$. Let
$$E=\bigcap_{i=1}^\infty \overline{B(E_i, h_i^d/(2\log h_i^{-1}))}.$$
Then the Hausdorff dimension of $E$ is at least $n/d$.
\end{lemma}

\begin{remark}\label{dimremark}
In fact, we will prove that for any ball $B(x,1)$ of unit radius, $E\cap B(x,1)$ has positive Hausdorff measure with respect to the gauge function
$r\mapsto r^{n/d} (\log r^{-1})^{n+1}$.
%
By modifying the assumption on the sequence $(h_i)$ in Lemma~\ref{dimlemma}, one can prove the same statement for any other gauge function $\phi$ satisfying
$$\lim_{r\to0}\phi(r)/r^{n/d}=\infty.$$
\end{remark}

The proof of Lemma~\ref{dimlemma} is straightforward but includes tedious calculations. 

\begin{proof}
\newcommand{\miez}[1]{\lceil #1 \rceil} 

Let $E'_i$ be a maximal subset of $E_i$ in which every two points have distance at least $h$. Then $B(E'_i, 2h)=\R^n$.

Let $b_i=h_i^d/(2\log h_i^{-1})$. Then $h_i^{d+1}/2\le b_i \le h_i^d /4$.

There exists a constant $C>0$, which depends only on $n$, such that for every $x\in\R^n$, $r>0$,
\beq
\#(B(x,r)\cap E'_i) \le C r^n h_i^{-n} \quad \text{if} \quad h_i/4\le r.
\eeq
Therefore
\beq\label{eq-metszi}
\#\{y \in E'_i \,:\, B(x,r) \cap B(y,b_i) \neq\emptyset \} \le C2^n r^n h_i^{-n} \quad \text{if} \quad h_i/4\le r.
\eeq
Similarly, there exists a constant $c>0$ which depends only on $n$ such that for every $x\in \R^n$, $r>0$,
\beq\label{eq-tart}
\#\{y \in E'_i \,:\, B(y,b_i) \subset B(x,r) \} \ge \miez{c r^n h_i^{-n}} \quad \text{if} \quad 3h_i\le r.
\eeq

Fix an arbitrary $x_0\in \R^n$, and let $E'_0=\{x_0\}$. Let $b_0=1$.

By \eqref{eq-tart}, for every $i\ge 0$ and $x\in E'_i$, there are at least $\miez{ cb_i^n h_{i+1}^{-n} }$ points $y\in E'_{i+1}$ such that $B(y, b_{i+1}) \subset B(x, b_i)$. By recursion we define sets $E''_i\subset E'_i$.
Let $E''_0=E'_0$. For $i\ge 0$, let $E''_{i+1}$ be a minimal subset of $E'_{i+1}$ such that for every $x\in E''_i$, there are exactly $\miez{cb_i^n h_{i+1}^{-n} }$ points $y\in E''_{i+1}$ with $B(y, b_{i+1}) \subset B(x, b_i)$. Thus
\beq\label{elemszam}
\# E''_k = \prod_{i=0}^{k-1} \miez{ c b_i^n h_{i+1}^{-n} }.
\eeq

Note that $3h_{i+1}\le b_i$, and $3h_1\le 1$.

Let $F_i=\overline{B(E''_i, b_i)}$ ($i=0, 1, \ldots$). These sets are finite unions of disjoint closed balls, and $F_0 \supset F_1 \supset \cdots$. Let $F=\bigcap F_i$. As $F\subset E$, it is enough to prove that $F$ has Hausdorff dimension at least $n/d$.
Let $F^k=\bigcap_{i=1}^k F_i$. Let $\lambda$ stand for the $n$-dimensional Lebesgue measure.

Note that $F_0$ is a ball of unit radius. By \eqref{elemszam},
\beq\label{legalabb}
\lambda(F^k)= b_k^n \lambda(F_0) \prod_{i=0}^{k-1} \miez{ c b_i^n h_{i+1}^{-n} }.
\eeq

Consider now a ball $B(x,r)$ with $r\le h_1/4$. Let $j$ be such that $h_{j+1}/4\le r\le h_j/4$. Then there exists at most one point $y\in E''_j$ for which $B(y,b_j)\cap B(x,r)\neq\emptyset$. Therefore there are at most $\miez{cb_j^n h_{j+1}^{-n}}$ points
$z\in E''_{j+1}$ for which $B(z,b_{j+1})\cap B(x,r)\neq\emptyset$. On the other hand, \eqref{eq-metszi} gives that there are at most $C2^n r^n h_{j+1}^{-n}$ points $z\in E''_{j+1}$ for which $B(z,b_{j+1})\cap B(x,r)\neq\emptyset$. From the definition of the sets $E''_i$ we obtain
\begin{align}\label{mdp1}
\lambda(B(x,r)\cap F^k) & \le  \#\{y\in E''_k : B(y,b_k)\cap B(x,r)\neq\emptyset \} \, b_k^n \lambda(F_0) \nonumber \\
& \le \min(\miez{cb_j^n h_{j+1}^{-n}}, \,C2^n r^n h_{j+1}^{-n}) \left( \prod_{i=j+1}^{k-1} \miez{c b_{i}^n h_{i+1}^{-n}} \right) b_k^n \lambda(F_0). 
\end{align}
for $k>j$. From \eqref{mdp1} and \eqref{legalabb},
\begin{align}
\frac{\lambda(B(x,r)\cap F^k)}{\lambda(F^k)} & \le
\min(\miez{cb_j^n h_{j+1}^{-n}}, \,C2^n r^n h_{j+1}^{-n}) 
\left(\prod_{i=1}^{j} \miez{cb_i^{n}h_{i+1}^{-n}}^{-1} \right)  \\
& \le \min(1, \, C2^n r^n h_{j+1}^{-n} \miez{cb_j^n h_{j+1}^{-n}}^{-1})
\left(\prod_{i=1}^{j-1} \miez{cb_i^{n}h_{i+1}^{-n}}^{-1}\right) \\
& \le \min(1, \, Cc^{-1} 2^n r^n b_j^{-n})
\left(\prod_{i=1}^{j-1} c^{-1} b_i^{-n}h_{i+1}^{n}\right) \\
& = \min(1, \, Cc^{-1} 2^n r^n b_j^{-n}) h_j^n
\left(\prod_{i=1}^{j-1} c^{-1} b_i^{-n}h_{i}^{n}\right) h_1^{-n} \\
& = \min(1, \, Cc^{-1} 2^n r^n b_j^{-n}) h_j^n c_{j-1}, \label{mdp2}
\end{align}
by seting $c_{j-1}=\left(\prod_{i=1}^{j-1} c^{-1} b_i^{-n}h_{i}^{n}\right) h_1^{-n}$.
Using $b_i\ge h_i^{d+1}/2$,
$$c_{j-1}\le c^{-j+1}\left(\prod_{i=1}^{j-1} 2^n h_i^{-dn}\right) h_1^{-n}=(2^n/c)^{j-1} h_1^{-n} \prod_{i=1}^{j-1} h_i^{-dn} \le C' \prod_{i=1}^{j-1} h_i^{-dn-1}$$
for some constant $C'$ depending on $c$. Therefore, by \eqref{tul2},
\beq\label{tul3}
c_{j-1}/C' \le \log h_j^{-1}
\eeq

First suppose that $h_j^d\le r \le h_j/4$. Then $c_{j-1}/C' \le \log  h_j^{-1} \le \log r^{-1}$, and \eqref{mdp2} implies
\begin{align}\label{egyik}
\frac{\lambda(B(x,r)\cap F^k)}{\lambda(F^k)} & \le
h_j^n c_{j-1} \le r^{n/d} c_{j-1} \le C' r^{n/d} \log r^{-1},
\end{align}
provided that $k>j$.

Now suppose that $h_{j+1}/4\le r \le h_j^d$. Then \eqref{mdp2} and \eqref{tul3} imply
\begin{align}
\frac{\lambda(B(x,r)\cap F^k)}{\lambda(F^k)} & \le
Cc^{-1} 2^n r^n b_j^{-n} h_j^n c_{j-1} \\
& \le Cc^{-1} 2^n c_{j-1}  2^n (\log h_j^{-1})^n h_j^{-dn} h_j^n  r^n\\
& \le Cc^{-1}4^n C' (\log h_j^{-1})^{n+1} h_j^{(1-d)n} r^n \\
& \le Cc^{-1}4^n C' (\log h_j^{-1})^{n+1} r^{n/d} \label{mekk}
\end{align}
as $r\le h_j^d$ and $d\ge 1$. Since $h_j^{-1} \le r^{-1/d}$, from \eqref{mekk} we obtain
\beq\label{masik}
\frac{\lambda(B(x,r)\cap F^k)}{\lambda(F^k)} \le Cc^{-1}4^n C' d^{-(n+1)} r^{n/d} (\log r^{-1})^{n+1},
\eeq
provided that $k>j$.
Combining \eqref{egyik} and \eqref{masik}, for every ball $B(x,r)$ with $r<h_1/4$, for $k$ large enough,
\beq\label{masik2}
\frac{\lambda(B(x,r)\cap F^k)}{\lambda(F^k)} \le C'' r^{n/d} (\log r^{-1})^{n+1},
\eeq
for some constant $C''$ (which depends only on $n$).

For each $k\ge 1$ we a define a Borel probability measure $\mu_k$ by setting $\mu_k(A)=\lambda(A\cap F^k) / \lambda(F^k)$ for every Borel set $A\subset \R^n$. Let $\mu$ be a weak limit of a convergent subsequence of $(\mu_k)$. Then $\mu$ is supported on $F$, and it satisfies $\mu(B(x,r))\le C'' r^{n/d} (\log r^{-1})^{n+1}$. The mass distribution principle implies that $F$ has Hausdorff dimension at least $n/d$. (In fact, we conclude that the Hausdorff measure of F corresponding to the gauge function $r\mapsto r^{n/d} (\log r^{-1})^{n+1}$ is positive.)
\end{proof}

\section{Open questions}\label{9}

The first question would be a natural strenghtening of Theorem~\ref{approxangles}.

\begin{question}\label{q1}
Is it true that for every $\alpha\in (0,\pi)$ there is a constant $c_\alpha<1$ such that for every $n\ge 2$, for every Borel set $E\subset \R^n$ with $\dim_H E >c_\alpha n$, $\alpha\in \iA(E)$?
\end{question}

Note that if $\dim_H E > n-1$ then $\iA(E)=[0,\pi]$, so Question~\ref{q1} is not relevant for  small $n$.

\begin{question}
Is there a constant $C$ such that every Borel set $E\subset\R^n$ of Hausdorff dimension $n/2+C$ contains every angle in $(0,\pi)$?
\end{question}

By Theorem~\ref{angledist}, $\iA(E)$ has positive Lebesgue measure if, for example, $C=4/3$. By Theorem~\ref{90}, there is a compact set  of dimension $n/2$ which does not contain the angle $\pi/2$.

\begin{question} Let $n\ge 2$.
Is it true that every Borel set $E\subset\R^n$ of Hausdorff dimension larger than $n/2$ contains the right angle? 
\end{question}

This is known to be true if $n$ is even, see \cite{hetes}.

\end{document}